\documentclass[a4paper,10pt,leqno]{amsart}
\usepackage{amsmath,amsfonts,amsthm,amssymb,dsfont}
\usepackage[alphabetic]{amsrefs}
\usepackage[OT4]{fontenc}
\usepackage{enumerate}
\usepackage{mathrsfs}
\usepackage{mathtools}
\usepackage{enumerate}
\usepackage{float}
\usepackage[dvipsnames]{xcolor}

\newtheorem{thm}{Theorem}[section]
\newtheorem{lem}[thm]{Lemma}
\newtheorem{prop}[thm]{Proposition}
\newtheorem{cor}[thm]{Corollary}
\newtheorem{theoremalph}{Theorem}

\theoremstyle{definition}
\newtheorem{rem}[thm]{Remark}
\newtheorem{defin}[thm]{Definition}

\makeatletter
\def\paragraph{\@startsection{paragraph}{4}%
  \z@\z@{-\fontdimen2\font}%
  {\normalfont\bfseries}}
\makeatother

\begin{document}

\title[Tits alternative for Artin groups of type FC]{Tits alternative for Artin groups of type FC}

\author[A.~Martin]{Alexandre Martin$^{\dag*}$}
           \address{Department of Mathematics and the Maxwell Institute for Mathematical Sciences,
           Heriot-Watt University,
           Riccarton,
           EH14 4AS Edinburgh, United Kingdom}
           \email{alexandre.martin@hw.ac.uk}
           \thanks{$\dag$ Partially supported by EPSRC New Investigator Award EP/S010963/1.}

\author[P.~Przytycki]{Piotr Przytycki$^{\ddagger*}$}
\address{
Department of Mathematics and Statistics,
McGill University,
Burnside Hall,
805 Sherbrooke Street West,
Montreal, QC,
H3A 0B9, Canada}

\email{piotr.przytycki@mcgill.ca}
\thanks{$\ddagger$ Partially supported by NSERC and National Science Centre, Poland UMO-2018/30/M/ST1/00668.}
\thanks{$*$ This work was partially supported by
the grant 346300 for IMPAN from the Simons Foundation and the matching
2015--2019 Polish MNiSW fund}
\maketitle

\begin{abstract}
Given a group action on a finite-dimensional $\mathrm{CAT}(0)$ cube complex, we give a simple criterion phrased purely in terms of cube stabilisers that ensures that the group satisfies the strong Tits alternative, provided that each vertex stabiliser satisfies the strong Tits aternative. We use it to prove that all Artin groups of type FC satisfy the strong Tits alternative. 
\end{abstract}

\section{Introduction}

The Tits alternative and its many variants, originating in the work of Tits \cite{Titsalternative}, deals with a striking dichotomy at the level of subgroups of a given group. A group satisfies the \textit{Tits alternative} if every finitely generated subgroup either contains a non-abelian free subgroup or is virtually soluble, and satisfies the \textit{strong Tits alternative} if this dichotomy holds also infinitely generated subgroups. Tits showed that linear groups in any characteristic satisfy the Tits alternative while linear groups in characteristic zero satisfy the strong Tits alternative \cite{Titsalternative}. Since then, many groups of geometric interest, and in particular groups displaying non-positively curved features, have been shown to satisfy the Tits alternative, including mapping class groups of hyperbolic surfaces \cite{Iv, McC}, outer automorphism groups of free groups \cite{BFH, BFH2}, groups of birational transformations of compact complex K\"ahler surfaces \cite{CantatTits}. The aim of this short note is to give a simple criterion to prove the Tits alternative for groups acting on finite-dimensional CAT(0) cube complexes, and to discuss applications, in particular to a large family of Artin groups, a class of groups generalising braid groups. 

Groups acting on CAT(0) cube complexes have a rich structure, and many tools have been developed over the years  in connection with the Tits alternative. Let us mention in particular that Sageev--Wise proved the strong Tits alternative for groups acting properly on CAT(0) cube complexes with a bound on the order of finite subgroups \cite{SW}. Caprace--Sageev \cite[Thm F]{CS} found a non-abelian free subgroup for groups acting on finite-dimensional CAT(0) cube complexes $X$ without a global fixed-point in $X\cup \partial X$, where $\partial X$ is the visual boundary. Fern\'os \cite[Thm~1.1]{F} proved the analogous result for groups without a finite orbit in $X\cup \partial X$ for $\partial X$ the Roller boundary. For groups acting on CAT(0) cube complexes, it is natural to ask whether the strong Tits alternative for all vertex stabilisers implies the strong Tits alternative for the whole group. This is however not the case, as already noted in \cite{SW}, see Remark~\ref{rem:exa} below. In order to obtain such a combination result, it is necessary to impose additional conditions on the group and the action. Such conditions do exist, and they generally require that for a particular class of subgroups of $G$, increasing sequences of subgroups eventually stabilise. This condition is on the finite subgroups of $G$ in the case of proper actions \cite{SW}, and on the finitely generated virtually soluble subgroups of $G$ in the general case, as was probably known to experts (see Corollary \ref{cor:Tits_chain_subgroups}). However, such conditions presuppose an understanding of the global structure of $G$ by requiring a control of subgroups of $G$ in a given class. It seems to us preferable to have a criterion that does not involve the global structure of $G$ but instead focuses solely on cube stabilisers. The main advantage of our criterion is thus to be formulated purely in local terms, i.e.\ in terms of the cube stabilisers. 

All the actions we consider are by cellular isometries. $\mbox{Stab}(C)$ denotes the setwise stabiliser of a cube $C$.

\begin{theoremalph}\label{cor:Tits_combination_local}
Let $G$ be a group acting on a finite-dimensional $\mathrm{CAT}(0)$ cube complex, such that each vertex stabiliser satisfies the strong Tits alternative, and we have the following property.
	\begin{description}
\item[$(*)$]	
	for every pair of intersecting cubes $C, C'$, there exists a cube $D$ containing~$C$ such that  
		$\mbox{Stab}(C) \cap \mbox{Stab}(C') = \mbox{Stab}(D)$.
\end{description}
Then $G$ satisfies the strong Tits alternative.
\end{theoremalph}

As an application, we  prove the strong Tits alternative for a large class of Artin groups. Let us first recall their definition. Let $S$ be a finite set. To every pair of elements $s\neq t\in S$, we associate $m_{st}=m_{ts}\in \{2,3,\ldots, \infty\}$. 
The associated \emph{Artin group} $A_S$ is given by the following presentation:
$$ A_S \coloneqq \langle S \ | \ \underbrace{sts\cdots}_{m_{st}}=\underbrace{tst\cdots}_{m_{st}}\rangle,$$  and the associated \emph{Coxeter group} $W_S$ is obtained by adding the relation $s^2=1$ for every $s\in S$. For a subet $S'$ of $S$, the subgroup of $A_S$ generated by $S'$ is isomorphic to $A_{S'}$ \cite{vdL}, so we think of $A_{S'}$ as a subgroup of $A_S$, and call it a \emph{standard parabolic subgroup}. An Artin group is said to be \textit{spherical} if $W_S$ is finite, and is \textit{of type FC} if for every subset $S' \subseteq S$ such that $m_{st} < \infty$ for every $s, t \in S'$,  the subgroup $A_{S'}$ is spherical. 

Artin groups have been the topic of intense research in recent years, with a common theme being to show that they enjoy many of the features of non-positively curved groups \cite{HKS, HO1, HO2, MP, Haettel2}. Several classes of Artin groups have been shown to satisfy the strong Tits alternative, including spherical Artin groups \cite{CW}, many two-dimensional Artin groups \cite{OP,MP}, and  Artin groups that are virtually cocompactly cubulated \cite{HJP, Haettel}.  In this note, we prove the following: 

\begin{theoremalph}\label{thm:Tits_FC}
Artin groups of type FC satisfy the strong Tits alternative.
\end{theoremalph}

In order to emphasise the wider applicability of our criterion, we also mention a new proof of a result of Antol\'{\i}n--Minasyan \cite{AM}, stating that the strong Tits alternative is stable under graph products. \\

\paragraph{Aknowledgements} We thank Pierre-Emmanuel Caprace, Talia Fern\'os, and Anthony Genevois for explaining to us the state of affairs in the subject. This work was conducted during  \emph{Nonpositive curvature} conference at Banach Center and \emph{LG\&TBQ} conference at the University of Michigan.

\section{First combination result}

In this section, we prove a first combination result for the strong Tits alternative, under a `global' condition on the action.

\begin{prop}\label{prop:Tits_combination}
	Let $G$ be a group acting on a finite-dimensional $\mathrm{CAT}(0)$ cube complex,  such that each vertex stabiliser satisfies the strong Tits alternative. Suppose that the poset of fixed-point sets of finitely generated virtually soluble subgroups of $G$ satisfies the \emph{descending chain condition}, i.e.\ every decreasing sequence
	$$ F_1 \supseteq F_2 \supseteq \cdots $$
	of fixed-point sets of finitely generated virtually soluble subgroups of $G$ satisfies $F_i = F_{i+1}$ for $i$ large enough.
	Then $G$ satisfies the strong Tits alternative.
\end{prop}

Note that we have the following immediate corollary phrased purely in terms of subgroups of $G$, which was probably folklore and  known to experts, although it does not seem to appear in the literature. 

\begin{cor}\label{cor:Tits_chain_subgroups}
		Let $G$ be a group acting on a finite-dimensional $\mathrm{CAT}(0)$ cube complex,  such that each vertex stabiliser satisfies the strong Tits alternative. Suppose that the poset of finitely generated virtually soluble groups of $G$ satisfies the \emph{ascending chain condition}, i.e.\ for every increasing sequence
		$$ H_1 \subseteq H_2 \subseteq\cdots $$
		of finitely generated virtually soluble subgroups of $G$, the inclusion $H_i\subseteq H_{i+1}$ is an isomorphism for $i$ large enough.
	Then $G$ satisfies the strong Tits alternative.
\end{cor}

\begin{rem}
\label{rem:exa}
	Note that Proposition \ref{prop:Tits_combination} and Corollary \ref{cor:Tits_chain_subgroups} do not hold if we only assume that all the vertex stabilisers satisfy the strong Tits alternative. Indeed, consider the case of the wreath product 
	$$G\coloneqq A_5 \wr \mathbb{Z} = \Big(  \bigoplus_{n \in \mathbb{Z}} A_5\Big) \rtimes \mathbb{Z},$$
	where $A_5$ denotes the alternating group on $5$ elements and $\mathbb{Z}$ acts on $ \bigoplus_{n \in \mathbb{Z}} A_5$ by shifting the indices. It is known that $G$ acts properly on a two-dimensional $\mathrm{CAT}(0)$ cube complex  \cite[Prop 9.33]{Gen}, and in particular vertex stabilisers satisfy the strong Tits alternative. However, $G$ itself does not satisfy the Tits alternative. Indeed, $G$ does not contain non-abelian free subgroups, and $G$ is not virtually soluble since finite index subgroups of $G$ contain a copy of the non-soluble group~$A_5$.   
\end{rem}
	
	\begin{proof}[Proof of Proposition \ref{prop:Tits_combination}]
		Let $H$ be a subgroup of $G$. Since $X$ is finite-dimensional, by \cite[Thm~1.1]{F} we have that $H$ contains a non-abelian free subgroup or virtually fixes a point in the Roller boundary of $X$. In the latter case, by \cite[Thm~B.1]{Cap} we have that $H$ admits a finite index subgroup~$H'$ that fits into a short exact sequence 
		$$ 1 \rightarrow N \rightarrow H' \rightarrow Q\rightarrow 1,$$
		where $Q$ is a finitely generated virtually abelian group of rank at most $\dim(X) $, and $N$ is a \emph{locally elliptic} subgroup of $G$, i.e.\ every finitely generated subgroup of~$N$ fixes a point of $X$.
		If $N$ contains a non-abelian free subgroup, then we are done. If $N$ does not contain a non-abelian free subgroup, consider the poset~$\mathcal{F}$ of fixed-point sets of finitely generated subgroups of~$G$ contained in~$N$. 
Since each finitely generated subgroup of~$N$ fixes a point of $X$, and hence satisfies the strong Tits alternative, it is virtually soluble. By the descending chain condition, there exists a smallest element $F\in \mathcal{F}$. Consequently for each element $g\in N$ we have $\mbox{Fix}(\langle g\rangle)\supseteq F$, and so $\mbox{Fix}(N) \supseteq F$.
Thus $\mbox{Fix}(N)$ is non-empty and so $N$ satisfies the strong Tits alternative. Since $N$ does not contain a non-abelian free subgroup, it is virtually soluble. Since the class of virtually soluble groups is stable under extensions, it now follows that $H'$, and hence $H$, is virtually soluble. 
	\end{proof}

\begin{rem} It follows from \cite[Thm~1.1]{F} and \cite[Thm~B.1]{Cap} that one can replace the first `strong Tits alternative' by `Tits alternative' in Corollary~\ref{cor:Tits_chain_subgroups}. However, we cannot do the same in Proposition~\ref{prop:Tits_combination}, since $N$ might not be finitely generated even if $H$ is finitely generated.

Similarly, as observed by Pierre-Emmanuel Caprace, since the family of amenable groups (resp.\ elementarily amenable groups) is closed under direct limits, if $G$ acts on a finite-dimensional $\mathrm{CAT}(0)$ cube complex, such that each finitely generated subgroup of each vertex stabiliser contains a non-abelian free subgroup or is amenable (resp.\ elementarily amenable), then each subgroup of $G$ contains a non-abelian free subgroup or is amenable (resp.\ elementarily amenable).
\end{rem}

\section{Local condition}

The descending chain condition for fixed-point sets appearing in Proposition \ref{prop:Tits_combination} is global in nature, as it requires to understand the virtually soluble subgroups of the groups under study (and their fixed-point sets). In this section we prove Theorem~\ref{cor:Tits_combination_local}, which involves the more tractable local condition $(*)$ that implies the descending chain condition. 

A poset $(\mathcal{F}, \leq)$ has \textit{height at most $n$} if every chain
$$ F_1 < F_2 < \cdots $$
of elements of $\mathcal{F}$ has length at most $n$. 

\begin{prop}\label{prop:local}
Let $G$ be a group acting 
on an $n$-dimensional $\mathrm{CAT}(0)$ cube complex $X$ satisfying property $(*)$ of Theorem~\ref{cor:Tits_combination_local}. Then the poset $\mathcal F$ of non-empty fixed-point sets of subgroups of $G$ has height at most $n+1$.
\end{prop}

In particular, Theorem~\ref{cor:Tits_combination_local} is a direct consequence of Propositions~\ref{prop:Tits_combination} and~\ref{prop:local}. We prove Proposition~\ref{prop:local} in several steps. 
The first one is the following `local to global' result:

\begin{lem}
\label{lemma:removeC}
Let $G$ be a group acting 
on a $\mathrm{CAT}(0)$ cube complex~$X$ satisfying property $(*)$.
Then property $(*)$ holds also for disjoint cubes $C,C'$.
\end{lem}
\begin{proof}
Let $C_0=C, C_1,\ldots, C_k=C'$ be the unique normal cube path from $C_0$ to $C_k$ (see \cite[\S3]{NR}). We prove by downward induction on $l=k-1,\ldots,0$ the claim that $\mbox{Stab}(C_l) \cap \mbox{Stab}(C_k) = \mbox{Stab}(D_l)$ for some $D_l\supseteq C_l$. For $l=k-1$ this is property~$(*)$. Now let $m<k-1$ and assume that we have proved the claim for $l=m+1$. We have $\mbox{Stab}(C_m) \cap \mbox{Stab}(C_k)=\mbox{Stab}(C_m) \cap \mbox{Stab}(C_{m+1})\cap \mbox{Stab}(C_k)$ by the uniqueness of normal paths. By the induction hypothesis we have $\mbox{Stab}(C_m) \cap \mbox{Stab}(C_{m+1})\cap \mbox{Stab}(C_k)=\mbox{Stab}(C_m)\cap \mbox{Stab}(D_{m+1})$ for some $D_{m+1}\supseteq C_{m+1}$. Thus by property~$(*)$ we have $D_m\supseteq C_m$ with $\mbox{Stab}(D_m)=\mbox{Stab}(C_m)\cap \mbox{Stab}(D_{m+1})$, proving the claim for $l=m$.
\end{proof}

\begin{cor}\label{lem:intersection}
Let $G$ be a group acting 
on an $n$-dimensional $\mathrm{CAT}(0)$ cube complex~$X$ satisfying property $(*)$.
Then the poset $\mathcal{P}$ of all cube stabilisers has height at most $n+1$.
\end{cor}

\begin{proof}
Suppose by contradiction that we have $P_1 \subsetneq P_2 \subsetneq \cdots \subsetneq P_{n+2}$ in $\mathcal{P}$. Let $C_{n+2}$ be a cube with $P_{n+2}=\mbox{Stab}(C_{n+2})$. Then by Lemma~\ref{lemma:removeC} we can choose a cube $C_{n+1}\supseteq C_{n+2}$
such that $\mbox{Stab}(C_{n+1})=\mbox{Stab}(C_{n+2})\cap P_{n+1}=P_{n+1}$. Analogously, we can inductively define $C_l$ for $l=n,\ldots, 1$ with $C_{l}\supseteq C_{l+1}$ and $\mbox{Stab}(C_l)=P_l$. Since the dimension of $X$ is $n$, for some $l$ we have $C_l=C_{l+1}$, contradicting $P_l\neq P_{l+1}$.
\end{proof}

\begin{rem} 
\label{rem:stable}
Note that if a poset $\mathcal{P}$ has finite height and the intersection of any two elements of $\mathcal{P}$ belongs to $\mathcal{P}$, then $\mathcal{P}$ is \textit{stable under all intersections}, that is, the intersection of any family of elements of $\mathcal{P}$ belongs to $\mathcal{P}$. 
\end{rem} 

\begin{defin}
Let $G$ be a group acting 
on a cube complex $X$.
Suppose that the poset $\mathcal{P}$ of all cube stabilisers is stable under all intersections, and let $H$ be a subgroup of~$G$ that fixes a point of $X$. We denote by $P_H\in\mathcal P$ the intersection of the nonempty family of the elements of $\mathcal{P}$ containing $H$. 
\end{defin}

\begin{lem}\label{lem:stab}
Let $G$ be a group acting 
on a cube complex $X$. Suppose that the poset of all cube stabilisers is stable under all intersections, and let $H$ be a subgroup of~$G$ that fixes a point of $X$. Then $\mbox{Fix}(H) = \mbox{Fix}(P_H)$.
\end{lem}

\begin{proof}
	Since $H \subseteq P_H$, we have $\mbox{Fix}(H) \supseteq \mbox{Fix}(P_H)$. Now, let $C$ be a cube of $\mbox{Fix}(H)$. Then $H \subseteq \mbox{Stab}(C)\in\mathcal P$, and so by the definition of $P_H$, we also have $P_H \subseteq \mbox{Stab}(C)$, or in other words $C \subseteq \mbox{Fix}(P_H)$. Thus, we get the reverse inclusion $\mbox{Fix}(H) \subseteq \mbox{Fix}(P_H)$.
\end{proof}

\begin{proof}[Proof of Proposition \ref{prop:local}]
Suppose by contradiction that we have a chain $\mbox{Fix}(H_1) \supsetneq \mbox{Fix}(H_2) \supsetneq \cdots \supsetneq \mbox{Fix}(H_{n+2})$ of non-empty fixed-point sets of subgroups of $G$. Without loss of generality we can assume $H_1 \subsetneq H_2 \subsetneq \cdots\subsetneq H_{n+2}$. By Lemma~\ref{lemma:removeC}, Corollary~\ref{lem:intersection}, and Remark~\ref{rem:stable}, the poset~$\mathcal{P}$ of all cube stabilisers has height at most $n+1$ and is stable under all intersections. Consider then the chain $P_{H_1} \subseteq P_{H_2} \subseteq \cdots \subseteq P_{H_{n+2}}$ of elements of $\mathcal{P}$. For some $k$ we have $P_{H_k} = P_{H_{k+1}}$. Lemma~\ref{lem:stab} implies that for $i=k,k+1$ we have $\mbox{Fix}(H_i) = \mbox{Fix}(P_{H_i})$, which contradicts $\mbox{Fix}(H_k)\neq \mbox{Fix}(H_{k+1})$.
\end{proof}

\section{Application: graph products and Artin groups of type FC}

As a first application, we recover the stability of the strong Tits alternative under graph products, as first proved by Antol\'{\i}n--Minasyan \cite{AM}.

\begin{defin}
Let $\Gamma$ be a simplicial graph with vertex set $V(\Gamma)$, and let  $\mathcal{G}=\{ G_v \mid v \in V(\Gamma) \}$ be a collection of groups.  The \emph{graph product} $\Gamma \mathcal{G}$ is defined as follows:
\begin{center}
	$\Gamma \mathcal{G} \coloneqq \left( \underset{v \in V(\Gamma)}{\ast} G_v \right) / \langle \langle gh=hg, \ h \in G_u, g \in G_v, \{ u,v \} ~\mbox{an edge of}~\Gamma \rangle \rangle$,
\end{center}
\end{defin}

For an induced subgraph $\Lambda$ of $\Gamma$, the subgroup of $\Gamma\mathcal{G}$ generated by $\{G_v\}_{v \in \Lambda}$  is isomorphic to $\Lambda\mathcal{G}$ \cite{G}. We thus think of $\Lambda\mathcal{G}$ as a subgroup of $\Gamma\mathcal{G}$.

We first recall the construction of a cube complex associated to a graph product of groups, introduced in \cite{Davis}. 

\begin{defin}
	The \textit{Davis complex} of the graph product $\Gamma \mathcal{G} $ is defined as follows: 
	\begin{itemize}
		\item Vertices correspond to left cosets  $g\Lambda \mathcal{G}$ for $g\in \Gamma\mathcal{G}$ and $\Lambda \subseteq \Gamma$ a complete subgraph of $\Gamma$.
		\item We add an edge between vertices $g\Lambda_1\mathcal{G}$ and $g\Lambda_2\mathcal{G}$ whenever $g\in \Gamma\mathcal{G}$ and $\Lambda_1 \subset  \Lambda_2$ are complete subgraphs of $ \Gamma$ that differ by exactly one vertex.
		\item More generally,  for $g\in \Gamma\mathcal{G}$ and $\Lambda_1 \subset  \Lambda_2$ complete subgraphs of $ \Gamma$ that differ by exactly $k$ vertices, we add a $k$-cube spanned by the vertices $g\Lambda\mathcal{G}$ for all complete subgraphs $\Lambda_1 \subseteq\Lambda \subseteq \Lambda_2$.
	\end{itemize}
	The group $\Gamma\mathcal{G}$ acts on the vertices by left multiplication of left cosets, and this action extends to the entire Davis complex.
\end{defin}

In \cite[Thm 5.1]{Davis}, Davis showed that the Davis complex associated to a graph product of groups is a $\mathrm{CAT}(0)$ cube complex.

\begin{prop}
	Let $\Gamma$ be a finite simplicial graph, and let $\mathcal{G} \coloneqq \{G_u\}_{u \in V(\Gamma)}$ be a collection of groups that satisfy the strong Tits alternative. Then the graph product $\Gamma\mathcal{G}$ satisfies the strong Tits alternative.
\end{prop}

\begin{proof}
To verify property~$(*)$ of Theorem \ref{cor:Tits_combination_local}, we can assume that $C_0=C\cap C'$ has vertices $\Lambda\mathcal{G}$ for all complete subgraphs $\Lambda_1 \subseteq \Lambda \subseteq \Lambda_2$ for some $\Lambda_1 \subseteq \Lambda_2$ complete subgraphs of $ \Gamma$. In particular, the stabiliser of $C_0$ is $\prod_{u \in V(\Lambda_1)} G_u$. Furthermore, since each $G_{u'}$, with $u'\in V(\Lambda_1)$, is normalised by $\prod_{u \in V(\Lambda_1)} G_u$, the stabilisers of cubes containing $C_0$ are exactly the direct products of the form $\prod_{u' \in V(\Lambda')} G_{u'}$ for $\Lambda'$ a complete subgraph of $\Lambda_1$. This  implies property~$(*)$ of Theorem \ref{cor:Tits_combination_local}. Moreover, since the strong Tits alternative is stable under finite direct products, the stabilisers of vertices of the Davis complex satisfy the strong Tits alternative. It now follows from Theorem \ref{cor:Tits_combination_local} that $\Gamma\mathcal{G}$  satisfies the strong Tits alternative. 
\end{proof}

We conclude this note by proving the strong Tits alternative for Artin groups of type FC.
Recall the following construction from \cite{CD}. 

\begin{defin}
	The \textit{Deligne cube complex} of an Artin group $A_S$ of type FC is the cube complex defined as follows: 
	\begin{itemize}
		\item Vertices correspond to left cosets  $gA_{S'}$ for $g\in A_S$ and $S'\subseteq S$ with $A_{S'}$ spherical.
		\item  We add an edge between vertices $gA_{S_1}$ and $gA_{S_2}$ whenever $g\in A_S$ and $S_1 \subset  S_2$ are subsets of $S$ that differ by exactly one element.
		\item More generally,  for $g\in A_S$ and $S_1 \subset  S_2$  subsets of $S$ that differ by exactly $k$~elements, with $A_{S_2}$ spherical, we add a $k$-cube spanned by the vertices $gA_{S'}$ for all $S_1 \subseteq S'\subseteq S_2$.
	\end{itemize}
	The group $A_S$ acts on the vertices by left multiplication of left cosets, and this action extends to the entire Deligne cube complex.
\end{defin}

In \cite[Thm 4.3.5]{CD}, Charney--Davis showed that the Deligne cube complex of an Artin group of type FC is a $\mathrm{CAT}(0)$ cube complex.

\begin{proof}[Proof of Theorem \ref{thm:Tits_FC}]
	Let $A_S$ be an Artin group of type FC. From the definition of the Deligne cube complex, we have that the stabilisers of vertices of $X$ are exactly the conjugates of the standard parabolic subgroups $A_{S'}$ that are spherical. Such Artin groups are known to be linear in characteristic zero \cite{CW}, and thus satisfy the strong Tits alternative.
	
	 For property~$(*)$ in Theorem \ref{cor:Tits_combination_local}, let $C_0=C\cap C'$ be a cube of the Deligne cube complex. After replacing $C_0$ by a cube in its $A_S$-orbit, we can assume that the vertices of~$C_0$ correspond to the cosets $A_{S'}$ over $S_1 \subseteq S' \subseteq S_2$ for some $S_1\subseteq S_2\subseteq S$ with $A_{S_2}$ spherical. In particular, the stabiliser of $C$ is $A_{S_1}$. To obtain property~$(*)$, it suffices to show that for subsets $S', S''\subseteq S_1$ and $g \in A_{S_1}$, the intersection $A_{S'} \cap gA_{S''}g^{-1}$ is a standard parabolic subgroup of $A_{S'}$. This is a consequence of \cite[proofs of Prop 7.2 and Thm 9.5]{CGGW}. 
	   It now follows from Theorem~\ref{cor:Tits_combination_local} that $A_S$ satisfies the strong Tits alternative. 
\end{proof}

Note that here property $(*)$ for arbitrary cubes $C,C'$, which follows from Lemma~\ref{lemma:removeC}, amounts to saying that the intersection of two conjugates of standard parabolic subgroups is again a conjugate of a standard parabolic subgroup. This was proved independently by Rose Morris-Wright \cite{M}.

\end{document}